\documentclass[11pt, openright]{article}

\usepackage[toc,title]{appendix}
\usepackage{amssymb}
\usepackage{amsfonts}
\usepackage{mathrsfs}
\usepackage{amsthm}
\usepackage[all]{xy}

\usepackage[hidelinks]{hyperref} 

\usepackage{enumerate}

\usepackage{xcolor}

\usepackage[utf8]{inputenc}
\usepackage{tikz-cd}
\usepackage{epsfig}
\usepackage[intlimits]{amsmath}

\newtheorem{thm}{Theorem}[section]
\newtheorem{cor}[thm]{Corollary}
\newtheorem{lem}[thm]{Lemma}
\newtheorem{prop}[thm]{Proposition}
\newtheorem{defn}[thm]{Definition}
\newtheorem{rem}[thm]{Remark}
\def\N{\mathbb N}

\def\Z{\mathbb Z}
\def\sC{\mathscr{C}}

\DeclareMathOperator{\Sym}{Sym}
\DeclareMathOperator{\dist}{dist_{\it S}}

\DeclareMathOperator{\Frob}{Frob}

\DeclareMathOperator{\traza}{tr}

\DeclareMathOperator{\id}{id}

\DeclareMathOperator{\HomDist}{HomDist_{\it S}}

\DeclareMathOperator{\defect}{def_{\it R}}

\DeclareMathOperator*{\D}{{\it D_{(S,R)}}}

\DeclareMathOperator*{\rateD}{{\it {D}_{\Gamma}}}
\DeclareMathOperator{\NC}{\left\langle  \left\langle R \right\rangle \right\rangle}

\usepackage{vmargin}

\setpapersize{A4}
\setmargins{2.5cm}       
{1.5cm}                        
{16.5cm}                      
{23.42cm}                    
{10pt}                           
{1cm}                           
{0pt}                             
{2cm}                           

\title{PROPERTY OF DEFECT DIMINISHING AND STABILITY}
\author{ Marco Antonio Garc\'ia Morales and Lev Glebsky}
\begin{document}
\maketitle
\begin{abstract}
We show that the defect diminishing is equivalent to the stability with linear rate. 
\end{abstract}

\section{Introduction}

The stability of a group $\Gamma$ (with respect to a class of groups $\sC$) means that any almost homomorphism to $\sC$ is
close to a homomorphism, see Definition~\ref{stability}. In \cite{1} the notion of defect diminishing was introduced, see Definition~\ref{def_dd1} and
Definition~\ref{def_dd2}. It was shown in \cite{1} that for some classes $\sC$ and $\Gamma$-modules $M$ the vanishing of the second cohomology
$H^2(\Gamma,M)$ implies the defect diminishing and that the defect diminishing implies stability. So, the defect diminishing is a kind of 
linear stability. 

In the present paper, we show that (under weaker assumptions) defect diminishing is equivalent to stability with a linear rate for finitely presented groups.
Particularly, this implies that there are stable groups that do not have defect diminishing. Indeed, O. Becker and J. Mosheiff \cite{Oren} showed that the rate of stability
of $\Z^d$ is polynomial (with respect to symmetric groups with normalized Hamming distance). 
It is worth mentioning that the stability of any abelian group (with respect to symmetric groups with normalized Hamming distance) was proven
by G. Arzhantseva and L. Paunescu in \cite{Goulnara_and_Liviu}.


\section{Stability}

Let $S$ be a finite set of symbols. We denote by $F(S)$ the free group on $S$. Let $R \subseteq F(S)$ be finite and $ \Gamma $ be a finitely presented group $\Gamma=\left\langle S\;|\;R\right\rangle= F(S)  /  \NC $ where $ \NC $ is the normal subgroup of $F(S)$ generated by $R$. Let $ \sC $ be a class of groups, all equipped with bi-invariant metric. Any map $ \phi: S \rightarrow G $, for a group $ G \in \sC $ uniquely determines a homomorphism $ F(S) \rightarrow G $ that we also denote by  $ \phi $. 

\begin{defn} \cite{1} \label{defdistHomDist} Let $G \in \sC$ and let $\phi, \psi: S \rightarrow G$ be maps. The defect of $\phi$ is defined by:
\begin{center}
$\defect (\phi)= \max\limits_{r \in R}d_G(\phi(r),1_G)$
\end{center}

The distance between $\phi$ and $\psi$ is defined by:
\begin{center}
$\dist(\phi,\psi)=\max\limits_{s \in S}d_G(\phi(s),\psi(s))$
\end{center}

The homomorphism distance of $\phi$ is defined by:
\begin{center}
$\HomDist(\phi)= \inf\limits_{\pi \in Hom(\Gamma, G)} \dist(\phi, \pi\restriction_S)$
\end{center}

Let $\langle \sC^S \rangle= \bigcup\limits_{G \in \sC } G^S$ where $G^S= \lbrace \phi: S \rightarrow G \rbrace$, that is, $\langle \sC^S \rangle$ are all possible maps $\phi: S \rightarrow G$ for $G \in \sC $.
\end{defn}

\begin{defn} \label{stability} \cite{Thom} A finitely presented group $\Gamma$ is called $ \sC $-stable if  for all $\epsilon >0 $ there exists $\delta>0$ such that 
 for all $\phi \in \langle \sC^S \rangle$ the inequality
$\defect (\phi) < \delta$ implies $\HomDist(\phi) < \epsilon$. Let us restate it to avoid ambiguity: 
$$
\forall \epsilon >0 \,\exists \delta>0 \,\forall \phi \in \langle \sC^S \rangle\;\left(
\defect(\phi) < \delta \Rightarrow \HomDist (\phi) < \epsilon \right). 
$$ 
\end{defn}

\begin{rem}The stability of $\Gamma$ does not depend on the particular choice of the presentation of the group $ \Gamma $ (see \cite{Goulnara_and_Liviu}): Tietze transformations preserve stability since the metric is bi-invariant. The stability of a group does depend on the class $ \sC $. 
\end{rem}

Interesting examples $ \sC =\{(G_n,d_n)\;|\;n\in \N\}$ are:

\begin{enumerate}[(1)]
\item \label{example1} $G_n = U(n)$, the group of Unitary $n\times n$ matrix.  The metric $d_n$ is induced by the normalized Hilbert-Schmidt norm $ \Vert A \Vert_{HS} = \sqrt{\frac{1}{n} \traza (A^* A) } $ ($d_n(A,B) = \Vert A - B \Vert$). 
\item \label{example2} $G_n = U(n)$, the metric $d_n$ is induced by the Schatten $p$-norm $ \Vert A \Vert_{p} =  \left(  \traza \vert T \vert^p \right) ^{\frac{1}{p}} $, where $\vert T \vert = \sqrt{T^*T}$. Note that if $p=2$ then  $ \Vert A \Vert_{2} = \Vert A \Vert_{\Frob}$.
\item \label{example3} $G_n = U(n)$, the metric $d_n$ is induced by the operator norm $ \Vert A \Vert_{op} = \sup\limits_{\Vert v \Vert =1} \Vert Av \Vert $ also known as Schatten $\infty$-norm.
\item \label{example4} $G_n=\Sym(n)$, the symmetric group of $n$ elements. $d_n$ is the normalized Hamming distance:
$d_n(\alpha,\beta)=\frac{1}{n}|\{j\;|\;\alpha(j)\neq\beta(j)\}|$.
\end{enumerate}


\subsection{Rate of stability}
The rate of stability is, roughly speaking, the dependence of $\epsilon$ and $\delta$ in Definition~\ref{stability}. See \cite{Oren} for details.
To make this precise we define the function 
$\D :\mathbb{R^+} \rightarrow \mathbb{R^+} $ as follows:

$$
\D (\delta)= \sup \limits_{\phi \in \langle \sC^S \rangle} \lbrace \HomDist (\phi) \mid \defect (\phi)< \delta \rbrace
$$

The function $\D $ is monotone increasing and depends on the presentation of the group $\Gamma$, but we show now that this dependence is just linear.\\

The following lemma is a reformulation of Definition \ref{stability}. The analogue of the lemma is used as the definition of stability in \cite{Ozawa_Thom}

\begin{lem} $ \displaystyle\lim_{\delta \rightarrow 0^{+}} \D (\delta) =0$ if and only if $\Gamma$ is $\sC$-stable. 
\end{lem}

%

Following O. Becker and J. Mosheiff, we define the rate stability $\rateD$ of the group $\Gamma$ as a class of functions (see Definition~\ref{stability_rate}).


\begin{defn} \label{def_precede}
  Let $f,g: (0, \delta_0] \rightarrow \mathbb{R^{+}}$ be monotone nondecreasing functions. Write $f \preceq g$ if $f(\delta) \leq C g(C\delta)+C\delta$ for some $C>0$ and all $\delta\in (0,\delta_0]$ for some $\delta_0>0$. We define the equivalence relation $\thicksim$ by saying that $f \thicksim g$ if and only if $f \preceq g$ and $g \preceq f $ (notice that the relation $\preceq$ is reflexive and transitive). Let $\left[ f \right]$ denote the class of $f$ with regard
      to this equivalence relation.  Clearly, $\preceq$ defines a partial order on equivalence classes: $\left[ f \right] \preceq \left[ g \right]$ if and only if
      $ f \preceq g $. 
\end{defn}

Note that if $f \preceq \id$ then $f(\delta) \leq M \delta$ for some $M$. Here $\id$ is an identical function: $\id(\delta)=\delta$.

\begin{prop} \label{prop_rate} \cite{Oren}
Let $\Gamma=\left\langle S\;|\;R\right\rangle$ be a finitely presented group. If $ \Gamma = \langle S'\;|\;R'  \rangle$ is another presentation of $\Gamma$. Then $D_{(S,R)} \sim D_{(S',R')}$.
\end{prop}

\begin{defn} \label{stability_rate}
Let $\Gamma=\left\langle S\;|\;R\right\rangle$ be a finitely presented group. The  rate stability $\rateD$ of the group $\Gamma$ is the equivalence class $\rateD = \left[ D_{(S,R)} \right]$.
\end{defn}

Proposition~\ref{prop_rate} implies that the rate of stability $\rateD$ of the finitely presented group $\Gamma$ does not depend on the presentation of $\Gamma$.


By the definition of $\thicksim$ the rate of stability $\rateD$ of a group $\Gamma$ can not be faster then linear.
The following lemma shows that it is not just by definition of $\thicksim$ but rather a natural phenomenon for non-free groups.

\begin{lem} \cite{Oren} \label{lem_lin}
  Let $ \Gamma = \langle S\;|\;R  \rangle $ be a finitely presented group with $R \neq \emptyset$, $R \neq 1_\Gamma$ and $\sC$ is the class of symmetric groups with the normalized Hamming
  distance. Then there exists $C>0$ and $\delta_0>0$ such that $C \delta \leq \D (\delta)$ for all $ \delta \in (0,\delta_0]$.
\end{lem}

By O. Becker and J. Mosheiff \cite{Oren} if $\sC$ is symmetric group with Hamming distance and $d=2,3,4...$ then
$O(\delta^\frac{1}{b})\preceq D_{\Z^d} \preceq O(\delta^\frac{1}{c})$  for any $b<2$ and some $c=c_d$, depending on $d$.



\section{Property of defect diminishing}

In this section we give the definition of the property of defect diminishing and a proof of the main theorem.

\begin{defn} An ultrafilter $\mathcal{U}$ on $\mathbb{N}$ is a collection of subsets of $\mathbb{N}$, such that:

\begin{enumerate}[(i)]
\item $A \in \mathcal{U}$ and $ A \subset B $ implies $B \in \mathcal{U} $
\item $A, B \in \mathcal{U} $ implies $ A \cap B \in \mathcal{U}$
\item $A \notin \mathcal{U}$ if, and only if $\mathbb{N} \setminus A \in \mathcal{U}$
\end{enumerate}
\end{defn}

We say that $\mathcal{U}$ is non-principal if $\lbrace n \rbrace \notin \mathcal{U} $ for every $n \in \mathbb{N}$. The existence of non-principal ultrafilters on $\mathbb{N}$ is ensured by the axiom of choice.
We fix a non-principal ultrafilter $\mathcal{U}$ on $\mathbb{N}$. Given a bounded sequence $(x_n )_{n \in \mathbb{N}}$ of real numbers we denote the limit along the ultrafilter by  $ \lim\limits_{n \to \mathcal{U}} x_n \in (- \infty, \infty)$. Formally, the limit is the unique $x \in \mathbb{R}$ such for all $\epsilon > 0$ we have $ \lbrace n \in \mathbb{N}: \mid x_n - x \mid < \epsilon  \rbrace \in \mathcal{U}$. For more information on ultrafilters and ultralimits see \cite{Ultra} appendix B. \\

We will use the notation Landau, let $(x_n )_{n \in \mathbb{N}}$ and $(y_n )_{n \in \mathbb{N}}$ be two sequences of positive real numbers, we denoted by $x_n= O_{\mathcal{U}}(y_n)$ if there exists $C>0$ such that $\lbrace n \mid x_n \leq C y_n \rbrace \in \mathcal{U}$. We denoted by $x_n=o_{\mathcal{U}}(y_n)$ if there is a third sequence of positive real numbers $\varepsilon_n$ such that $ \lim\limits_{n \to \mathcal{U}} \varepsilon_n = 0$ and $x_n= \varepsilon_n y_n$.

\begin{defn} \cite{1} A sequence of maps $\phi_n : S \rightarrow G_n$, for $(G_n,d_n) \in \sC $ is called an asymptotic homomorphism to $\sC$ if 
\begin{center}
$ \lim\limits_{n \to \mathcal{U}} \defect (\phi_n)=0$
\end{center}
\end{defn}

\begin{defn}\label{def_dd1}
Let $\phi_n: S \rightarrow G_n$ with $G_n \in \sC$ be an asymptotic homomorphism, we say that an asymptotic homomorphism $\phi'_n: S \rightarrow G_n$ diminishes the defect of $ (\phi_n)_{n \in \mathbb{N}}$ if:

\begin{enumerate}[(a)]
\item $\dist(\phi_n,\phi'_n) = O_{\mathcal{U}}(\defect(\phi_n$))
\item $\defect(\phi'_n)= o_{\mathcal{U}}(\defect(\phi_n))$
\end{enumerate}

We say that $ (\phi_n)_{n \in \mathbb{N}}$  has the property of defect diminishing if there is an asymptotic homomorphism $(\phi'_n)_{n \in \mathbb{N}}$ that diminishes the defect of $ (\phi_n)_{n \in \mathbb{N}}$.  
\end{defn}

\begin{defn}\label{def_dd2} The group $\Gamma$ has the property of defect diminishing (with respect to $\sC$) if every asymptotic homomorphism to $\sC$ has the property of defect diminishing. 
\end{defn}


\begin{thm} \label{thm:neat}
  Let $\Gamma= \langle S \mid R \rangle$ be a finitely presented group and $\sC$ a class of groups such that each $(G,d) \in \sC$ is a complete metric space.
  Then the group $\Gamma$ has the property of defect diminishing if and only if $\D \preceq \id $. 
\end{thm}

\begin{cor}
Let $\Gamma$ be a finitely presented group and $\sC$ a class of groups such that each $(G,d) \in \sC $ is a complete metric space. The group $\Gamma$ has the property of defect diminishing if and only if $D_{\Gamma} \preceq \left[ \id \right]$.
\end{cor}

\begin{cor}
The property of defect diminishing does not depend on the particular choice of the presentation of the group $ \Gamma $.
\end{cor} 

\begin{proof}
  If $\sC$ is a class of groups such that each $(G,d) \in \sC$ is a complete metric space, the proof follows from the Proposition~\ref{prop_rate} and
  Theorem~\ref{thm:neat}. General case may be proved directly similarly to Proposition~\ref{prop_rate}. 
\end{proof}

For the proof of Theorem~\ref{thm:neat} we need the following proposition.


\begin{prop} \label{proposition}
If $\Gamma=\langle S \mid R \rangle$ has the property of defect diminishing then there exists $M, \varepsilon \in \mathbb{R^+}$ such that for all 
$G \in \sC$ and $\phi \in G^S$ with $\defect(\phi)<\epsilon$ there exists $\psi \in G^S$ such that:

\begin{center}
\begin{enumerate}
\item $\defect(\psi) < \frac{1}{2} \defect (\phi)$. 
\item $\dist(\phi , \psi) < M \defect(\phi)$.
\end{enumerate}
\end{center}

\end{prop}

\begin{proof} 

Suppose that the conclusion of the proposition is false. Then for every $n \in \mathbb{N}$ there is $\phi_n \in (G_n)^S $ with $G_n \in \sC $ and $\defect(\phi_n)<\frac{1}{n}$, such that every $\psi \in (G_n)^S $ with  $\defect (\psi) < \frac{1}{2} \defect (\phi_n)$ satisfies  $\dist(\phi_n , \psi) \geq n \defect(\phi_n)$.

So we have an asymptotic homomorphism $(\phi_n)_{n \in \mathbb{N}}$ that does not have the property of defect diminishing. Therefore, $\Gamma$ does not have the property of defect diminishing.

\end{proof}
                        

\begin{proof} [Proof of Theorem \ref{thm:neat}]
Suppose that $\D\preceq \id $, that is, there exists $M>0$ and $\delta_0 > 0$ such that $\forall  0< \delta <\delta_0$ we have that $\D (\delta)< M \delta$. 
Let $(\phi_n)_{n \in \mathbb{N}}$ be an asymptotic homomorphism and $\epsilon_n= \defect (\phi_n)$. By the definition of asymptotic homomorphism 
$ \displaystyle\lim_{n \rightarrow \mathcal{U}} \epsilon_n =0$. Let $X=\{n\;|\;\epsilon_n<\delta_0\}$.
For $n\in X$ we have that $\HomDist(\phi_n) < M\epsilon_n$ by Definition \ref{defdistHomDist} and there is a $\pi_n \in Hom(\Gamma,G_n)$ 
that complies $\dist(\phi_n, \pi_n\restriction_S) < M \defect(\phi_n)$. Define $\phi'_n=\pi_n$ for $n\in X$ and $\phi'_n=\phi_n$ for $n\not\in X$.
Then $\phi'_n$ diminishing the defect of $\phi_n$ as $X\in\mathcal{U}$. \\


Suppose that the group $\Gamma$  has the property of defect diminishing. We apply Proposition~\ref{proposition}. Let $M, \varepsilon \in \mathbb{R^+}$ 
be as in Proposition~\ref{proposition}.  
Let $\phi \in G^S $ be with $\defect(\phi)<\varepsilon$. Inductively we may construct a sequence of  maps $\phi_j \in G^S$,  $\phi_0=\phi$, such that
$\defect (\phi_j) < \frac{1}{2} \defect (\phi_{j-1}) < \frac{\varepsilon}{2^{j}}$ and  
$\dist(\phi_j , \phi_{j-1}) < M \defect (\phi_{j-1})<M\defect(\phi)\frac{1}{2^{j-1}}$. 
It follows that  $(\phi_n)_{n \in \mathbb{N}}$ is a Cauchy sequence. Let $\phi_\infty$ be its limit point ($(G^S, \dist)$ is a complete metric space as  
$(G,d) \in \sC $ is). We can check that $\phi_\infty$ is a homomorphism and $\dist(\phi,\phi_\infty)<2M\defect(\phi)$.
It follows that $\D (\delta)<2M \delta$ for $\delta<\varepsilon$. Therefore, $\D \preceq \id$.

\end{proof}


\bibliography{References}
\bibliographystyle{alpha}

\end{document}